\documentclass{amsart}
\usepackage{amscd}
\usepackage{amsmath,empheq}
\usepackage{amsfonts}
\usepackage{amssymb}
\usepackage{mathrsfs}
\usepackage[all]{xy}
\usepackage{mathtools}
\newtheorem{theorem}{Theorem}[section]

\newtheorem{prop}[theorem]{Proposition}
\newtheorem{remark}{Remark}[section]
\newtheorem{corollary}[theorem]{Corollary}

\newtheorem{definition}[theorem]{Definition}
\newenvironment{proof-sketch}{\noindent{\bf Sketch of Proof}\hspace*{1em}}{\qed\bigskip}

\everymath{\displaystyle}
\newcommand{\RR}{\mathbb R}
\newcommand{\FF}{\mathbb F}
\newcommand{\NN}{\mathbb N}

\newcommand{\di}{\displaystyle}
\newcommand{\ep}{\varepsilon}
\newcommand{\ri}{\rightarrow}
\newcommand{\wt}{W^{1,\theta}(\Omega)}
\newcommand{\intom}{\int_\Omega}
\newcommand{\bb}{\begin{equation}}
\newcommand{\bbb}{\end{equation}}

\renewcommand{\leq}{\leqslant}

\renewcommand{\geq}{\geqslant}
\begin{document}
\title[Double-phase Robin problems]{Existence and multiplicity of solutions\\ for double-phase Robin problems}
\author[N.S. Papageorgiou]{N.S. Papageorgiou}
\address[N.S. Papageorgiou]{ Department of Mathematics,  National Technical University, 
				Zografou Campus, Athens 15780, Greece \& Institute of Mathematics, Physics and Mechanics, 1000 Ljubljana, Slovenia}
\email{\tt npapg@math.ntua.gr}
\author[V.D. R\u{a}dulescu]{V.D. R\u{a}dulescu}
\address[V.D. R\u{a}dulescu]{Faculty of Applied Mathematics, AGH University of Science and Technology, al. Mickiewicza 30, 30-059 Krak\'{o}w, Poland  \& Department of Mathematics, University of Craiova, 200585 Craiova, Romania \& Institute of Mathematics, Physics and Mechanics, 1000 Ljubljana, Slovenia}
\email{\tt radulescu@inf.ucv.ro}
\author[D.D. Repov\v{s}]{D.D. Repov\v{s}}
\address[D.D. Repov\v{s}]{Faculty of Education and Faculty of Mathematics and Physics, University of Ljubljana \& Institute of Mathematics, Physics and Mechanics, 1000 Ljubljana, Slovenia}
\email{\tt dusan.repovs@guest.arnes.si}
\begin{abstract}
We consider a double phase Robin problem with a Carath\'eodory nonlinearity. When the reaction is superlinear but without satisfying the Ambrosetti-Rabinowitz condition, we prove an existence theorem. When the reaction is resonant, we prove a multiplicity theorem. Our approach is Morse theoretic, using the notion of homological local linking.
\end{abstract}
\maketitle

\section{Introduction}
Let $\Omega\subseteq \RR^N$ be a bounded domain with Lipschitz boundary $\partial\Omega$. In this paper we study the
following two phase Robin problem
\begin{equation}\label{eq1}
	\left\{
		\begin{array}{lll}
	&\di -{\rm div}\, (a_0(z)|Du|^{p-2}Du)-\Delta_qu+\xi(z)|u|^{p-2}u=f(z,u)&\, \mbox{in}\ \Omega		\\
		&\di \frac{\partial u}{\partial n_\theta}+\beta(z)|u|^{p-2}u=0&\, \mbox{on}\ \partial\Omega,
		\end{array}
	\right\}
\end{equation}
where $1<q<p\leq N$.

In this problem, the weight $a_0:\overline\Omega\ri\RR$ is Lipschitz continuous and $a_0(z)>0$ for all $z\in\Omega$. The potential function $\xi\in L^\infty(\Omega)$ satisfies $\xi(z)\geq0$ for a.a. $z\in\Omega$, while the reaction term $f(z,x)$ is Carath\'eodory (that is, for all $x\in\RR$ the mapping $z\mapsto f(z,x)$ is measurable and for a.a. $z\in\Omega$ the function $x\mapsto f(z,x)$ is continuous).
Let $F(z,\cdot)$ be the primitive of $f(z,\cdot)$, that is, $F(z,x)=\int_0^xf(z,s)ds$. We assume that for a.a. $z\in\Omega$, $F(z,\cdot)$ is $q$-linear near the origin. On the other hand, near $\pm\infty$, we consider two distinct cases for $f(z,\cdot)$:\\
(i) for a.a. $z\in\Omega$, $f(z,\cdot)$ is $(p-1)$-superlinear but without satisfying the Ambrosetti-Rabinowitz condition (the AR-condition for short), which is common in the literature when dealing with superlinear problems;\\
(ii) for a.a. $z\in\Omega$, $f(z,\cdot)$ is $(p-1)$-linear and possibly resonant with respect to the principal eigenvalue of the weighted $p$-Laplacian
$$u\mapsto -{\rm div}\, (a_0(z)|Du|^{p-2}Du)$$
with Robin boundary condition.

In the boundary condition, $\frac{\partial u}{\partial n_\theta}$ denotes the conormal derivative of $u$ corresponding to the modular function $\theta(z,x)=a_0(z)x^p+x^q$
for all $z\in\Omega$, all $x\geq 0$. We interpret this derivative via the nonlinear Green identity (see Papageorgiou, R\u adulescu and Repov\v{s} \cite[p. 34]{18}) and
$$\frac{\partial u}{\partial n_\theta}=[a_0(z)|Du|^{p-2}+|Du|^{q-2}]\,\frac{\partial u}{\partial n}\ \mbox{for all}\ u\in C^1(\overline\Omega),$$
with $n(\cdot)$ being the outward unit normal on $\partial\Omega$. The boundary coefficient $\beta$ satisfies $\beta\in C^{0,\alpha}(\partial\Omega)$ with $0<\alpha<1$
and $\beta(z)\geq 0$ for all $z\in\partial\Omega$.

The differential operator in problem \eqref{eq1} is a weighted $(p,q)$-Laplace operator and it corresponds to the energy functional
$$u\mapsto \int_\Omega [a_0(z)|Du|^p+|Du|^q]dz.$$
Since we do not assume that the weight function $a_0(z)$ is bounded away from zero, the continuous integrand $\theta_0:\Omega\times\RR^N\ri\RR_+$ of this integral functional
exhibits unbalanced growth, namely
$$|y|^q\leq \theta_0(z,y)\leq c_0(1+|y|^p)\ \mbox{for a.a. $z\in \Omega$, all $y\in\RR^N$ and some $c_0>0$}.$$

Such integral functionals were first investigated by Marcellini \cite{14} and Zhikov \cite{22}, in connection with problems in nonlinear elasticity theory. Recently, Baroni, Colombo and Mingione   \cite{3} and Colombo and Mingione \cite{6,7} revived the interest in them and produced important local regularity results for the minimizers of such functionals. A global regularity theory for such problems remains elusive.

In this paper, using tools from Morse theory (in particular, critical groups), we prove an existence theorem (for the superlinear case) and a multiplicity theorem (for the linear resonant case). Existence and multiplicity results for two phase problems were proved recently by Cencelj, R\u adulescu and Repov\v{s} \cite{4} (problems with variable growth), Colasuonno and Squassina \cite{5} (eigenvalue problems), Liu and Dai \cite{12} (existence of solutions for problems with superlinear reaction), Papageorgiou, R\u adulescu and Repov\v s \cite{19} (multiple solutions for superlinear problems), and Papageorgiou, Vetro and Vetro \cite{20} (parametric Dirichlet problems). The approach in all the aforementioned works is different and the hypotheses on the reaction are more restrictive.

Finally, we mention that $(p,q)$-equations arise in many mathematical models of physical processes. We refer to the very recent works of Bahrouni, R\u adulescu and Repov\v s \cite{1,2} and the references therein.

\section{Mathematical background}
The study of two-phase problems requires the use of Musielak-Orlicz spaces. So, let $\theta:\Omega\times \RR_+\ri\RR_+$ be the modular function defined by
$$\theta (z,x)=a_0(z)x^p+x^q\ \mbox{for all}\ z\in \Omega,\ x\geq 0.$$
This is a generalized N-function (see Musielak \cite{16}) and it satisfies
$$\theta (z,2x)\leq 2^p\theta (z,x)\ \mbox{for all}\ z\in\Omega,\ x\geq 0,$$
that is, $\theta(z,\cdot)$ satisfies the ($\Delta_2$)-property (see Musielak \cite[p. 52]{16}). Using the modular function $\theta(z,x)$, we can define the Musielak-Orlicz space
$L^\theta(\Omega)$ as follows:
$$L^\theta(\Omega)=\left\{u:\Omega\ri\RR;\ \mbox{$u$  is measurable and}\ \int_\Omega\theta (z,|u|)dz<\infty\right\}.$$

This space is equipped with the so-called ``Luxemburg norm" defined by
$$\|u\|_\theta=\inf\left\{\lambda>0:\ \int_\Omega \theta (z,\frac{|u|}{\lambda})dz\leq 1\right\}.$$

Using $L^\theta(\Omega)$, we can define the following Sobolev-type space $W^{1,\theta}(\Omega)$, by setting
$$W^{1,\theta}(\Omega)=\{u\in L^\theta(\Omega):\ |Du|\in L^\theta(\Omega)\}.$$
We equip $W^{1,\theta}(\Omega)$ with the norm $\|\,\cdot\,\|$ defined by
$$\|u\|=\|u\|_\theta +\|Du\|_\theta ,$$
where $\|Du\|_\theta =\|\, |Du|\,\|_\theta$. The spaces $L^\theta (\Omega)$ and $W^{1,\theta}(\Omega)$ are separable and uniformly convex (hence reflexive) Banach spaces.

Let $\hat\theta(z,x)$ be another modular function. We say that ``$\hat\theta$ is weaker than $\theta$" and write $\hat\theta\prec\theta$, if there exist $c_1,\, c_2>0$ and a function $\eta\in L^1(\Omega)$ such that
$$\hat\theta (z,x)\leq c_1\,\theta (z,c_2x)+\eta(z)\ \mbox{for a.a.}\ z\in\Omega\ \mbox{and all}\ x\geq 0.$$
Then we have
$$L^\theta(\Omega)\hookrightarrow L^{\hat\theta}(\Omega)\ \mbox{and}\ W^{1,\theta}(\Omega)\hookrightarrow W^{1,\hat\theta}(\Omega)\ \mbox{continuously}.$$

Combining this fact with the classical Sobolev embedding theorem, we obtain the following embeddings; see Propositions 2.15 and 2.18 of Colasuonno and Squassina \cite{5}.

\begin{prop}\label{prop1} We assume that $1<q<p<\infty$. Then the following properties hold.

(a) If $q\not=N$, then $W^{1,\theta}(\Omega)\hookrightarrow L^r(\Omega)$ continuously for all $1\leq r\leq q^*$, where
$$q^*=\left\{ \begin{array}{lll} &\di\frac{Nq}{N-q}&\quad\mbox{if}\ q<N\\
&\di +\infty &\quad\mbox{if}\ q\geq N.\end{array}\right.
$$

(b) If $q=N$, then $W^{1,\theta}(\Omega)\hookrightarrow L^r(\Omega)$ continuously for all $1\leq r<\infty$.

(c) If $q\leq N$, then $W^{1,\theta}(\Omega)\hookrightarrow L^r(\Omega)$ compactly for all $1\leq r< q^*$.

(d) If $q> N$, then $W^{1,\theta}(\Omega)\hookrightarrow L^\infty(\Omega)$ compactly.

(e)  $W^{1,\theta}(\Omega)\hookrightarrow W^{1,q}(\Omega)$ continuously.
\end{prop}

We have
$$L^p(\Omega)\hookrightarrow L^\theta(\Omega)\hookrightarrow L^p_{a_0}(\Omega)\cap L^q(\Omega)$$
with both embeddings being continuous.

We consider the modular function
$$\rho_\theta (u)=\int_\Omega \theta (z,|Du|)dz=\int_\Omega [a_0(z)|Du|^p+|Du|^q]dz\ \mbox{for all}\ u\in W^{1,\theta}(\Omega).$$

There is a close relationship between the norm $\|\,\cdot\,\|$ of $W^{1,\theta}(\Omega)$ and the modular functional $\rho_\theta(\cdot)$; see Proposition 2.1 of Liu and Dai \cite{12}.

\begin{prop}\label{prop2}
(a) If $u\not=0$, then $\|Du\|_\theta =\lambda$ if and only if $\rho_\theta (\frac u\lambda)\leq 1$.

(b) $\|Du\|_\theta <1$ (resp. $=1,\, >1$) if and only if $\rho_\theta (u)<1$ (resp. $=1,\, >1$).

(c) If $\|Du\|_\theta <1$, then $\|Du\|_\theta^p\leq \rho_\theta (u)\leq \|Du\|_\theta^q$.

(d)  If $\|Du\|_\theta >1$, then $\|Du\|_\theta^q\leq \rho_\theta (u)\leq \|Du\|_\theta^p$.

(e) $\|Du\|_\theta\ri 0$ if and only if $\rho_\theta(u)\ri 0$.

(f) $\|Du\|_\theta\ri +\infty$ if and only if $\rho_\theta(u)\ri +\infty$.
\end{prop}

On $\partial\Omega$ we consider the $(N-1)$-dimensional Hausdorff (surface) measure $\sigma (\cdot)$. Using this measure, we can define in the usual way the ``boundary" Lebesgue spaces $L^s(\partial\Omega)$ for $1\leq s\leq \infty$. It is well-known that there exists a unique continuous linear map $\gamma_0:W^{1,q}(\Omega)\ri L^q(\partial\Omega)$, known as the ``trace map", such that
$$\gamma_0(u)=u|_{\partial\Omega}\ \mbox{for all}\ u\in W^{1,q}(\Omega)\cap C(\overline\Omega).$$
We have
$${\rm im}\, \gamma_0=W^{\frac{1}{q'},q}(\Omega)\ \left(\frac 1q+\frac{1}{q'}=1\right)\ \mbox{and}\ {\rm ker}\, \gamma_0=W_0^{1,q}(\Omega).$$
Moreover, the trace map $\gamma_0(\cdot)$ is compact into $L^s(\partial\Omega)$ for all $1\leq s<(N-1)q/(N-q)$ if $q<N$, and for all $1\leq s<\infty$ if $q\geq N$. In what follows, for the sake of notational simplicity, we drop the use of the trace map $\gamma_0(\cdot)$. All restrictions of the Sobolev functions on the boundary $\partial\Omega$ are understood in the sense of traces.

Let $\langle\,\cdot,\cdot\,\rangle$ denote the duality brackets for the pair $(W^{1,\theta}(\Omega),W^{1,\theta}(\Omega)^*)$ and $\langle\,\cdot,\cdot\,\rangle_{1,q}$
denote the duality brackets for the pair $(W^{1,q}(\Omega),W^{1,q}(\Omega)^*)$. We introduce the maps $A_p^{a_0}:W^{1,\theta}(\Omega)\ri W^{1,\theta}(\Omega)^*$ and
$A_q:W^{1,q}(\Omega)\ri W^{1,q}(\Omega)^*$ defined by
$$\langle A_p^{a_0}(u),h\rangle =\int_\Omega a_0(z)|Du|^{p-2}(Du,Dh)_{\RR^N}dz\ \mbox{for all}\ u,h\in W^{1,\theta}(\Omega),$$
$$\langle A_q(u),h\rangle_{1,q} =\int_\Omega |Du|^{q-2}(Du,Dh)_{\RR^N}dz\ \mbox{for all}\ u,h\in W^{1,q}(\Omega).$$
We have
$$\langle A_q(u),h\rangle_{1,q} =\langle A_q(u),h\rangle\ \mbox{for all}\ u,h\in W^{1,\theta}(\Omega).$$

We introduce the following hypotheses on the weight $a_0(\cdot)$ and on the coefficients $\xi(\cdot)$ and $\beta(\cdot)$.

\smallskip
$H_0$: $a_0:\overline\Omega\ri\RR$ is Lipschitz continuous, $a_0(z)>0$ for all $z\in\Omega$, $\xi\in L^\infty(\Omega)$, $\xi(z)\geq 0$ for a.a. $z\in\Omega$, $\beta\in C^{0,\alpha}(\partial\Omega)$ with $0<\alpha<1$, $\xi\not\equiv0$ or $\beta\not\equiv0$ and $q>Np/(N+p-1)$.

\begin{remark}\label{rem1} The latter condition on the exponent $q$ implies that $W^{1,\theta}(\Omega)\hookrightarrow L^p(\partial\Omega)$ compactly and $q<p^*$.
\end{remark}

We introduce the $C^1$-functional $\gamma_p:W^{1,\theta}(\Omega)\ri\RR$ defined by
$$\gamma_p(u)=\int_\Omega a_0(z)|Du|^pdz+\int_\Omega \xi(z)|u|^pdz+\int_{\partial\Omega}\beta(z)|u|^pd\sigma\ \mbox{for all}\ u\in W^{1,\theta}(\Omega).$$

Then hypotheses $H_0$, Lemma 4.11 of Mugnai and Papageorgiou \cite{15}, and Proposition 2.4 of Gasinski and Papageorgiou \cite{10}, imply that
\begin{equation}\label{eq2}c_1\,\|u\|^p\leq\gamma_p(u)\ \mbox{for some $c_1>0$, all $u\in W^{1,\theta}(\Omega)$}.\end{equation}

We denote by $\hat\lambda_1(p)$ the first (principal) eigenvalue of the following nonlinear eigenvalue problem
\begin{equation}\label{eq3}
	\left\{
		\begin{array}{lll}
	&\di -{\rm div}\, (a_0(z)|Du|^{p-2}Du)+\xi(z)|u|^{p-2}u=\hat\lambda |u|^{p-2}u&\, \mbox{in}\ \Omega		\\
		&\di \frac{\partial u}{\partial n_p}+\beta(z)|u|^{p-2}u=0&\, \mbox{on}\ \partial\Omega.
		\end{array}
	\right\}
\end{equation}

Here, $\frac{\partial u}{\partial n_p}=|Du|^{p-2}\frac{\partial u}{\partial n}$. The eigenvalue $\hat\lambda_1(p)$ has the following variational characterization
\begin{equation}\label{eq4}\hat\lambda_1(p)=\inf\left\{\frac{\gamma_p(u)}{\|u\|^p_p}:\ u\in W^{1,p}(\Omega)\setminus\{0\} \right\}\ \mbox{(see \cite{17})}.\end{equation}

Then by \eqref{eq2}, we see that $\hat\lambda_1(p)>0$. This eigenvalue is simple (that is, if $\hat u,\, \hat v$ are corresponding eigenfunctions, then $\hat u=\eta\hat v$ with $\eta\in\RR\setminus\{0\}$) and isolated (that is, if $\hat\sigma (p)$ denotes the spectrum of \eqref{eq3}, then we can find $\ep>0$ such that $(\hat\lambda_1(p),\hat\lambda_1(p)+\ep)\cap \hat\sigma (p)=\emptyset$). The infimum in \eqref{eq4} is realized on the corresponding one-dimensional eigenspace, the elements of which have fixed sign. We denote by $\hat u_1(p)$ the corresponding positive, $L^p$-normalized (that is, $\|\hat u_1(p)\|_p=1$) eigenfunction. We know that $\hat u_1(p)\in L^\infty(\Omega)$ (see Colasuonno and Squassina \cite[Section 3.2]{5}) and $\hat u_1(p)(z)>0$ for a.a. $z\in\Omega$ (see Papageorgiou, Vetro and Vetro \cite[Proposition 4]{19}).

We will also use the spectrum of the following nonlinear eigenvalue problem
$$-\Delta_qu=\hat\lambda |u|^{q-2}u\ \mbox{in}\ \Omega,\quad \frac{\partial u}{\partial n}=0\ \mbox{on}\ \partial\Omega.$$

It is well known that this problem  has a sequence of variational eigenvalues $\{\hat\lambda_k(q)\}_{k\geq 1}$ such that
$\hat\lambda_k(q)\ri +\infty$ as $k\ri\infty$. We have $\hat\lambda_1(q)=0<\hat\lambda_2(q)$ (see Gasinski and Papageorgiou \cite[Section 6.2]{9}).

Let $X$ be a Banach space and $\phi\in C^1(X,\RR)$. We denote by $K_\phi$ the critical set of $\phi$, that is,
$$K_\phi=\{u\in X:\ \phi'(u)=0\}.$$
Also, if $\eta\in\RR$, then we set
$$\phi^\eta=\{u\in X:\ \phi(u)\leq\eta\}.$$

Consider a topological pair $(A,B)$ such that $B\subseteq A\subseteq X$. Then for every $k\in \NN_0$, we denote by $H_k(A,B)$ the $k$th-singular homology group for the pair $(A,B)$ with coefficients in a field $\FF$ of characteristic zero (for example, $\FF=\RR$). Then each $H_k(A,B)$ is an $\FF$-vector space and we denote by ${\rm dim}\, H_k(A,B)$ its dimension. We also recall that the homeomorphisms induced by maps of pairs and the boundary homomorphism $\partial$, are all $\FF$-linear.

Suppose that $u\in K_\phi$ is isolated. Then for every $k\in\NN_0$, we define the ``$k$-critical group" of $\phi$ at $u$ by
$$C_k(\phi, u)=H_k(\phi^c\cap U, \phi^c\cap U\setminus\{u\}),$$
where $U$ is an isolating neighborhood of $u$, that is, $K_\phi\cap U\cap\phi^c=\{u\}$. The excision property of singular homology implies that this definition is independent of the choice of the isolating neighborhood $U$.

We say that $\phi$ satisfies the ``C-condition" if it has the following property:
\begin{center}{``Every sequence $\{u_n\}_{n\geq 1}\subseteq X$ such that $\{\phi(u_n)\}_{n\geq 1}\subseteq\RR$ is bounded and $(1+\|u_n\|)\phi'(u_n)\ri 0$ in $X^*$ as $n\ri\infty$, has a strongly convergent subsequence". }\end{center}

Suppose that $\phi\in C^1(X,\RR)$ satisfies the C-condition and that $\inf \phi(K_\phi)>-\infty$. Let $c<\inf \phi(K_\phi)$. Then the critical groups of $\phi$ at infinity are defined by
$$C_k(\phi,\infty)=H_k(X,\phi^c)\ \mbox{for all}\ k\in\NN_0.$$
On account of the second deformation theorem (see Papageorgiou, R\u adulescu and Repov\v s \cite[p. 386]{18}, Theorem 5.3.12) this definition is independent of the choice of the level $c<\inf \phi(K_\phi)$.

Our approach is based on the notion of local $(m,n)$-linking ($m,n\in\NN$), see Papageorgiou, R\u adulescu and Repov\v s \cite[Definition 6.6.13, p. 534]{18}.

\begin{definition}\label{defi3} Let $X$ be a Banach space, $\phi\in C^1(X,\RR)$, and $0$ an isolated critical point of $\phi$ with $\phi(0)=0$. Let $m,n\in\NN$. We say that $\phi$ has a ``local $(m,n)$-linking" near the origin if there exist a neighborhood $U$ of $0$ and nonempty sets $E_0,\, E\subseteq U$, and $D\subseteq X$ such that $0\not\in E_0\subseteq E$, $E_0\cap D=\emptyset$ and
\\
(a) $0$ is the only critical point of $\phi$ in $\phi^0\cap U$;\\
(b) ${\rm dim}\, {\rm im}\, i_*-{\rm dim}\, {\rm im}\, j_*\geq n$, where
$$i_*:H_{m-1}(E_0)\ri H_{m-1}(X\setminus D)\ \mbox{and}\ j_*:H_{m-1}(E_0)\ri H_{m-1}(E)$$
are the homomorphisms induced by the inclusion maps
$i:E_0\ri X\setminus D$ and $j:E_0\ri E$;\\
(c) $\phi|_E\leq 0<\phi|_{U\cap D\setminus \{0\}}.$
\end{definition}

\begin{remark}\label{rem2} The notion of ``local $(m,n)$-linking" was introduced by Perera \cite{21} as a generalization of the concept of local linking due to Liu \cite{11}. Here we introduce a slightly more general version of this notion.\end{remark}

\section{Superlinear case}
In this section we treat the superlinear case, that is, we assume that the reaction $f(z,\cdot)$ exhibits $(p-1)$-superlinear growth near $\pm\infty$.

The hypotheses on $f(z,x)$ are the following.

\smallskip
$H_1$: $f:\Omega\times\RR\ri\RR$ is a Carath\'eodory function such that $f(z,0)=0$ for a.a. $z\in\Omega$ and\\
(i) $|f(z,x)|\leq \hat a(z)(1+|x|^{r-1})$ for a.a. $z\in\Omega$ and all $x\in\Omega$, with $\hat a\in L^\infty(\Omega)$, $p<r<q^*$;\\
(ii) if $F(z,x)=\int_0^x f(z,s)ds$, then $\di \lim_{x\ri\pm\infty}\frac{F(z,x)}{|x|^p}=+\infty$ uniformly for a.a. $z\in\Omega$;\\
(iii) if $\eta(z,x)=f(z,x)x-pF(z,x)$, then there exists $e\in L^1(\Omega)$ such that
$$\eta(z,x)\leq \eta(z,y)+e(z)\ \mbox{for a.a. $z\in\Omega$ and all $0\leq x\leq y$ or $y\leq x\leq 0$;}$$
(iv) there exist $\delta>0$, $\theta\in L^\infty(\Omega)$ and $\hat\lambda>0$ such that
$$0\leq\theta(z)\ \mbox{for a.a. $z\in\Omega$, $\theta\not\equiv 0$, $\hat\lambda\leq\hat\lambda_2(q)$},$$
$$\theta(z)|x|^q\leq qF(z,x)\leq\hat\lambda |x|^q\ \mbox{for a.a. $z\in\Omega$ and all $|x|\leq\delta$}.$$

\begin{remark}\label{rem3} Evidently, hypotheses $H_1(ii),\, (iii)$ imply that for a.a. $z\in\Omega$, the function $f(z,\cdot)$ is superlinear. However, to express this superlinearity, we do not invoke the usual AR-condition. We recall that the AR-condition says that there exist $\tau>p$ and $M>0$ such that
\begin{equation}\label{eq5a}0<\tau F(z,x)\leq f(z,x)x\ \mbox{for a.a. $z\in\Omega$ and all $|x|\geq M$; and}\end{equation}
\begin{equation}\label{eq5b}0<{\rm essinf}_\Omega\, F(\cdot, \pm M). \end{equation}
Integrating \eqref{eq5a} and using \eqref{eq5b}, we obtain a weaker condition, namely
$$\begin{array}{lll}
& c_2|x|^\tau\leq F(z,x) &\quad \mbox{for a.a. $z\in\Omega$, all $|x|\geq M$ and some $c_2>0$,}\\
\Rightarrow & c_3|x|^\tau\leq f(z,x)x &\quad \mbox{for a.a. $z\in\Omega$, all $|x|\geq M$ and with $c_3=\tau c_2>0$.}
\end{array}$$

Therefore the AR-condition implies that, eventually, $f(z,\cdot)$ has at least $(\tau-1)$-polynomial growth.
\end{remark}

In the present work, instead of the AR-condition, we use the quasimonotonicity hypothesis $H_1(iii)$, which is less restrictive and incorporates in our framework also $(p-1)$-superlinear nonlinearities with slower growth near $\pm\infty$ (see the examples below). Hypothesis $H_1(iii)$ is a slight generalization of a condition which can be found in Li and Yang \cite{13}. There are very natural ways to verify the quasimonotonicity condition. So, if there exists $M>0$ such that for a.a. $z\in\Omega$, either the function
$$x\mapsto \frac{f(z,x)}{|x|^{q-2}x}\ \mbox{is increasing on $x\geq M$ and decreasing on $x\leq-M$}$$
or the mapping
$$x\mapsto \eta(z,x)\ \mbox{is increasing on $x\geq M$ and decreasing on $x\leq-M$,}$$
then hypothesis $H_1(iii)$ holds.

Hypothesis $H_1(iv)$ implies that for a.a. $z\in\Omega$, the primitive $F(z,\cdot)$ is $q$-linear near 0.

\smallskip
{\it Examples.} The following functions satisfy hypotheses $H_1$. For the sake of simplicity we drop the $z$-dependence:
$$\begin{array}{ll}
&\di f_1(x)=\left\{\begin{array}{lll}
&\di \mu |x|^{q-2}x&\ \mbox{if}\ |x|\leq 1\\
&\di \mu |x|^{r-2}x&\ \mbox{if}\ |x|> 1\ \ \mbox{(with $0<\mu\leq\hat\lambda_2(q)$ and $p<r<q^*$)}\end{array}\right.\\
&\ \\
&\di f_2(x)=\left\{\begin{array}{lll}
&\di \mu |x|^{q-2}x&\ \mbox{if}\ |x|\leq 1\\
&\di \mu |x|^{p-2}x\ln x+\mu |x|^{\tau-2}x&\ \mbox{if}\ |x|> 1\ \ \mbox{(with $0<\mu\leq\hat\lambda_2(q)$ and $1<\tau<p$).}\end{array}\right.\\
\end{array}
$$

Note that only $f_1$ satisfies the AR-condition, whereas the function $f_2$ does not satisfy this growth condition.

The energy functional for problem \eqref{eq1} is the $C^1$-functional $\varphi:W^{1,\theta}(\Omega)\ri\RR$ defined by
$$\varphi (u)=\frac 1p\,\gamma_p(u)+\frac 1q\, \|Du\|^q_q-\int_\Omega F(z,u)dz\ \mbox{for all}\ u\in W^{1,\theta}(\Omega).$$

Next, we show that $\varphi (\cdot)$ satisfies the C-condition.

\begin{prop}\label{prop4}
If hypotheses $H_0,\, H_1$ hold, then the functional $\varphi(\cdot)$ satisfies the C-condition.
\end{prop}

\begin{proof}
We consider a sequence $\{u_n\}_{n\geq 1}\subseteq W^{1,\theta}(\Omega)$ such that
\begin{equation}\label{eq6} |\varphi (u_n)|\leq c_4\ \mbox{for some $c_4>0$ and all $n\in\NN$},\end{equation}
\begin{equation}\label{eq7} (1+\|u_n\|)\varphi'(u_n)\ri 0\ \mbox{in $W^{1,\theta}(\Omega)^*$ as $n\ri\infty$}.\end{equation}

From \eqref{eq7} we have
\begin{equation}\label{eq8}  \begin{split}
\di  \Big|\langle A_p^{a_0}(u_n),h\rangle  +\langle A_q(u_n),h\rangle &\di +\intom \xi(z)|u_n|^{p-2}u_nhdz+\int_{\partial\Omega}\beta(z)|u_n|^{p-2}u_nhd\sigma  \\  &\di   -  \intom f(z,u_n)hdz \Big|\leq\frac{\ep_n\|h\|}{1+\|u_n\|},  \end{split}   \bbb
for all $h\in \wt$, with $\ep_n\ri 0$.

In \eqref{eq8} we choose $h=u_n\in\wt$ and obtain for all $n\in\NN$
\bb\label{eq9} -\intom a_0(z)|Du_n|^pdz-\|Du_n\|^q_q-\intom \xi(z)|u_n|^pdz-\int_{\partial\Omega}\beta(z)|u_n|^pd\sigma +\intom f(z,u_n)u_ndz\leq\ep_n.\bbb
Also, by \eqref{eq6} we have for all $n\in\NN$,
\bb\label{eq10} \intom a_0(z)|Du_n|^pdz+\frac pq\,\|Du_n\|^q_q+\frac{p}{q}\intom \xi(z)|u_n|^pdz+\frac{p}{q}\int_{\partial\Omega}\beta(z)|u_n|^pd\sigma
-\intom pF(z,u_n)dz\leq pc_4.\bbb
We add relations \eqref{eq9} and \eqref{eq10}. Since $q<p$, we obtain
\bb\label{eq11}\intom\eta(z,u_n)dz\leq c_5\ \mbox{for some $c_5>0$ and all $n\in\NN$.}\bbb

\smallskip {\it Claim.} The sequence $\{u_n\}_{n\geq 1}\subseteq\wt$ is bounded.

\smallskip We argue by contradiction. Suppose that the claim is not true. We may assume that
\bb\label{eq12}\|u_n\|\ri\infty\ \mbox{as $n\ri\infty$}.\bbb
We set $y_n=u_n/\|u_n\|$ for all $n\in\NN$. Then $\|y_n\|=1$ and so we may assume that
\bb\label{eq13}y_n \xrightarrow{w} y \ \mbox{in $\wt$ and } y_n\ri y\ \mbox{in}\ L^r(\Omega)\ \mbox{and in}\ L^p(\partial\Omega),\bbb
see hypotheses $H_0$, Proposition \ref{prop1} and Remark \ref{rem1}.

We first assume that $y\not\equiv 0$. Let
$$\Omega_+=\{z\in\Omega:\ y(z)>0\}\ \mbox{and}\ \Omega_-=\{z\in\Omega:\ y(z)<0\}.$$
Then at least one of these measurable sets has positive Lebesgue measure on $\RR^N$. We have
$$u_n(z)\ri+\infty\ \mbox{for a.a. $z\in\Omega_+$ and}\ u_n(z)\ri-\infty\ \mbox{for a.a. $z\in\Omega_-$.}$$

Let $\hat\Omega=\Omega_+\cup\Omega_-$ and let $|\,\cdot\,|_N$ denote the Lebesgue measure on $\RR^N$. We see that $|\hat\Omega|_N>0$ and on account of hypothesis $H_1(ii)$, we have
\bb\label{eq14}\begin{array}{ll}
&\di \frac{F(z,u_n(z))}{\|u_n\|^p}=\frac{F(z,u_n(z))}{|u_n(z)|^p}\,|y_n(z)|^p\ri +\infty\ \mbox{for a.a. $z\in\hat\Omega$,}\\
& \ \\
\Rightarrow &\di \int_{\hat\Omega}\frac{F(z,u_n(z))}{\|u_n\|^p}dz\ri +\infty\ \mbox{by Fatou's lemma}.\end{array}\bbb
Hypotheses $H_1(i),\, (ii)$ imply that
\bb\label{eq15} F(z,x)\geq -c_6\ \mbox{for a.a. $z\in\Omega$, all $x\in\RR$ and some $c_6>0$.}\bbb
Thus we obtain
\bb\label{eq16}
\begin{array}{ll}
\di \intom \frac{F(z,u_n)}{\|u_n\|^p}dz&\di =\int_{\hat\Omega}\frac{F(z,u_n)}{\|u_n\|^p}dz+\int_{\Omega\setminus\hat\Omega}\frac{F(z,u_n)}{\|u_n\|^p}dz\\
&\di\geq \int_{\hat\Omega}\frac{F(z,u_n)}{\|u_n\|^p}dz-\frac{c_6|\Omega|_N}{\|u_n\|^p}\ \mbox{(see \eqref{eq15})},\\
&\di \Rightarrow  \lim_{n\ri\infty}\int_{\Omega}\frac{F(z,u_n)}{\|u_n\|^p}dz=+\infty\ \mbox{(see \eqref{eq14} and \eqref{eq12})}.\end{array}
\bbb
By \eqref{eq6}, we have
\bb\label{eq17}
\intom\frac{pF(z,u_n)}{\|u_n\|^p}dz\leq\gamma_p(y_n)+\frac pq\, \frac{1}{\|u_n\|^{p-q}}\,\|Dy_n\|^q_q+\frac{c_4}{\|u_n\|^p}\leq c_7,\bbb
for some $c_7>0$ and all $n\in\NN$ (see \eqref{eq12} and recall that $\|y_n\|=1$).

We compare relations \eqref{eq14} and \eqref{eq17} and arrive at a contradiction.

\smallskip
Next, we assume that $y=0$. Let $\mu>0$ and set $v_n=(p\mu)^{1/p}y_n$ for all $n\in\NN$. Evidently, we have
\bb\label{eq18}
\begin{array}{ll}
&\di v_n\ri 0\ \mbox{in $L^r(\Omega)$ (see \eqref{eq13})},\\
\Rightarrow &\di \intom F(z,v_n)dz\ri 0\ \mbox{as $n\ri\infty$}.\end{array}\bbb

Consider the functional $\psi:\wt\ri\RR$ defined by
$$\psi(u)=\frac 1p\, \gamma_p(u)-\intom F(z,u)dz\ \mbox{for all}\ u\in\wt.$$

Clearly, $\psi\in C^1(\wt,\RR)$ and
\bb\label{eq19}\psi\leq\varphi.\bbb
We can find $t_n\in [0,1]$ such that
\bb\label{eq20} \psi (t_nu_n)=\min\{\psi (tu_n):\ 0\leq t\leq 1\}\ \mbox{for all}\ n\in\NN.\bbb

Because of \eqref{eq12}, we can find $n_0\in\NN$ such that
\bb\label{eq21} 0<\frac{(p\mu)^{1/p}}{\|u_n\|}\leq 1\ \mbox{for all}\ n\geq n_0.\bbb
Therefore
$$\begin{array}{ll}
\psi(t_nu_n)&\di\geq \psi(v_n)\ \mbox{(see \eqref{eq20}, \eqref{eq21})}\\
&\geq \mu\gamma_p(y_n)-\intom F(z,v_n)dz\\
&\di\geq \mu c_1-\intom F(z,v_n)dz\ \mbox{(see \eqref{eq2} and recall that $\|y_n\|=1$)}\\
&\geq \frac \mu 2\, c_1\ \mbox{for all $n\geq n_1\geq n_0$ (see \eqref{eq18})}.\end{array}$$
Since $\mu>0$ is arbitrary, it follows that
\bb\label{eq22} \psi (t_nu_n)\ri +\infty\ \mbox{as}\ n\ri\infty.\bbb
Note that
\bb\label{eq23}\psi(0)=0\ \mbox{and}\ \psi(u_n)\leq c_4\ \mbox{for all}\ n\in\NN\ \mbox{(see \eqref{eq6}, \eqref{eq19})}.\bbb
By \eqref{eq22} and \eqref{eq23} we can infer that
\bb\label{eq24} t_n\in (0,1)\ \mbox{for all}\ n\geq n_2.\bbb
From \eqref{eq20} and \eqref{eq24}, we can see that for all $n\geq n_2$ we have
\bb\label{eq25}\begin{array}{ll}
0&=\di t_n\frac{d}{dt}\psi(tu_n)|_{t=t_n}\\
& \ \\
&\di =\langle \psi'(t_nu_n),t_nu_n\rangle\ \mbox{(by the chain rule)}\\
&\di =\gamma_p(t_nu_n)-\intom f(z,t_nu_n)(t_nu_n)dz.
\end{array}\bbb
It follows that
$$0\leq t_nu_n^+\leq u_n^+\ \mbox{and}\ -u_n^-\leq -t_nu_n^-\leq 0\ \mbox{for all}\ n\in\NN$$
(recall that $u_n^+=\max\{u_n,0\}$ and $u_n^-=\max\{-u_n,0\}$).

By hypothesis $H_1(iii)$, we have
$$\eta (z,t_nu_n^+)\leq\eta (z,u_n^+)+e(z)\ \mbox{for a.a.}\ z\in \Omega \ \mbox{and all}\ n\in\NN,$$
$$\eta (z,-t_nu_n^-)\leq\eta (z,-u_n^-)+e(z)\ \mbox{for a.a.}\ z\in \Omega \ \mbox{and all}\ n\in\NN.$$
From these two inequalities and since $u_n=u_n^+-u_n^-$, we obtain
\bb\label{eq26}\begin{array}{ll}
&\di\eta (z,t_nu_n)\leq\eta (z,u_n)+e(z)\ \mbox{for a.a.}\ z\in\Omega\ \mbox{and all}\ n\in\NN,\\
\Rightarrow &\di f(z,t_nu_n)(t_nu_n)\leq\eta(z,u_n)+e(z)+pF(z,t_nu_n)\ \mbox{for a.a.}\ z\in\Omega\ \mbox{and all}\ n\in\NN.\end{array}\bbb
We return to \eqref{eq25} and apply \eqref{eq26}. Then
\bb\label{eq27}\begin{array}{ll}
&\di\gamma_p(t_nu_n)-p\intom F(z,t_nu_n)dz\leq\intom \eta (z,u_n)dz+\|e\|_1\ \mbox{for all}\ n\in\NN,\\
\Rightarrow &\di p\psi (t_nu_n)\leq c_8\ \mbox{for some $c_8>0$ and all $n\in\NN$ (see \eqref{eq11}.}\end{array}\bbb
We compare \eqref{eq22} and \eqref{eq27} and arrive at a contradiction.

This proves the claim.

On account of this claim, we may assume that
\bb\label{eq28} u_n\xrightarrow {w} u\ \mbox{in}\ \wt\ \mbox{and}\ u_n\ri u\ \mbox{in}\ L^r(\Omega)\ \mbox{and in}\ L^p(\partial\Omega)\bbb
(see hypotheses $H_0$).

From \eqref{eq28} we have
\bb\label{eq29} Du_n\ri Du\ \mbox{in}\ L^p_{a_0}(\Omega,\RR^N)\quad\mbox{and}\quad
Du_n(z)\rightarrow Du(z)\ \mbox{a.a.}\ z\in\Omega.\bbb
In \eqref{eq8} we choose $h=u_n-u\in\wt$, pass to the limit as $n\ri\infty$ and use \eqref{eq29} and the monotonicity of $A_p(\cdot)^{a_0}$.
We obtain
$$\begin{array}{ll} &\di\limsup_{n\ri\infty}\langle A_p^{a_0}(u_n),u_n-u\rangle \leq 0,\\
\Rightarrow &\di \limsup_{n\ri\infty}\|Du_n\|_{L^p_{a_0}(\Omega,\RR^N)}\leq\|Du\|_{L^p_{a_0}(\Omega,\RR^N)}.\end{array}$$
On the other hand, from \eqref{eq29} we have
$$\liminf_{n\ri\infty}\|Du_n\|_{L^p_{a_0}(\Omega,\RR^N)}\geq\|Du\|_{L^p_{a_0}(\Omega,\RR^N)}.$$
Therefore we conclude that
\bb\label{eq30}\|Du_n\|_{L^p_{a_0}(\Omega,\RR^N)}\ri\|Du\|_{L^p_{a_0}(\Omega,\RR^N)}.\bbb

The space $L^p_{a_0}(\Omega,\RR^N)$ is uniformly convex, hence it has the Kadec-Klee property (see Papageorgiou, R\u adulescu and Repov\v s \cite[Remark 2.7.30, p. 127]{18}). So, it follows from \eqref{eq29} and \eqref{eq30}  that
$$\begin{array}{ll}
&\di Du_n\ri Du \ \mbox{in}\ L^p_{a_0}(\Omega,\RR^N),\\
\Rightarrow &\di Du_n\ri Du\ \mbox{in}\ L^q(\Omega,\RR^N) \ \mbox{since $L^p_{a_0}(\Omega,\RR^N)\hookrightarrow L^q(\Omega,\RR^N)$ continuously},\\
\Rightarrow &\di \rho_\theta (|Du_n-Du|)\ri 0\ \mbox{(see Proposition \ref{prop2})},\\
\Rightarrow &\di \|u_n-u\|\ri 0\ \mbox{(see \eqref{eq28} and Proposition \ref{prop2})},\\
\Rightarrow &\di\mbox{$\varphi$ satisfies the C-condition}.\end{array}$$
The proof is now complete.
\end{proof}

\begin{prop}\label{prop5} If hypotheses $H_0$, $H_1$ hold, then the functional $\varphi(\cdot)$ has a local $(1,1)$-linking at $0$.\end{prop}

\begin{proof}
Since the critical points of $\varphi$ are solutions of problem \eqref{eq1}, we may assume that $K_\varphi$ is finite or otherwise we already have infinitely many nontrivial solutions of \eqref{eq1} and so we are done.

Choose $\rho\in(0,1)$ so small that $K_\varphi\cap\bar{B}_\rho=\{0\}$ (here, $B_\rho=\{u\in\wt:\ \|u\|<\rho\}$).
Let $V=\RR$ and let $\delta>0$ as postulated by hypothesis $H_1(iv)$. Recall that on a finite dimensional normed space all norms are equivalent. So, by taking $\rho\in(0,1)$ even smaller as necessary, we have
\bb\label{eq31}
\|u\|\leq\rho\Rightarrow |u|\leq\delta\ \mbox{for all}\ u\in V=\RR.\bbb
Then for $u\in V\cap\bar{B}_\rho$, we have
$$\begin{array}{ll}
\varphi (u)&\di \leq \frac 1p\,\gamma_p(u)-\frac{|u|^q}{q}\intom \theta(z)dz\ \mbox{(see \eqref{eq31} and Hypothesis $H_1(iv)$)}\\
&\di =\frac{|u|^p}{p}\left(\intom\xi(z)dz+\int_{\partial\Omega}\beta(z)d\sigma\right)-\frac{|u|^q}{q}\intom \theta(z)dz\\
&\di\leq c_9\|u\|^p-c_{10}\|u\|^q\ \mbox{for some $c_9,c_{10}>0$ (see hypotheses $H_0$ and $H_1(iv)$)}.\end{array}$$
Since $q<p$, choosing $\rho\in(0,1)$ small, we conclude that
\bb\label{eq32} \varphi|_{V\cap\bar B_\rho}\leq 0.\bbb

Let $$D=\{u\in\wt:\ \|Du\|^q_q\geq\hat\lambda_2(q)\|u\|^q_q\}.$$
For all $u\in D$ we have
$$\begin{array}{ll}
\di\varphi (u)&\di =\frac 1p\,\gamma_p(u)+\frac 1q\,\|Du\|^q_q-\int_{\{|u|\leq\delta \}}F(z,u)dz-\int_{\{|u|>\delta \}}F(z,u)dz\\
&\di\geq \frac 1p\,\gamma_p(u)+\frac 1q\left(\|Du\|^q_q-\intom \hat\lambda |u|^qdz\right) -\intom F(z,u)dz\\
&\qquad\qquad\qquad \mbox{(see hypotheses $H_1(iv)$)}\\
&\di\geq \frac 1p\,\gamma_p(u)+\frac 1q \intom (\hat\lambda_2(q)-\hat\lambda)|u|^qdz-c_{11}\|u\|^r\\
&\qquad\qquad\qquad \mbox{for some $c_{11}>0$ (since $u\in D$ and see hypotheseis $H_1(iv)$)}\\
&\di\geq\frac{c_{11}}{p}\,\|u\|^p-c_{11}\,\|u\|^r\ \mbox{(see \eqref{eq21})}.
\end{array}$$
Since $p<r$, for small $\rho\in (0,1)$  we have
\bb\label{eq33} \varphi|_{D\cap\bar B_\rho\setminus\{0\}}>0.\bbb

Let $U=\bar B_\rho$, $E_0=V\cap\partial B_\rho$, $E=V\cap\bar B_\rho$ and $D$ as above. We have $0\notin E_0$, $E_0\subseteq E\subseteq U=\bar B_\rho$ and $E_0\cap D=\emptyset$
(see Definition \ref{defi3}).

Let $Y$ be the topological complement of $V$. We have that
$$\wt=V\oplus Y \ \mbox{(see \cite[pp. 73, 74]{18})}.$$
So, every $u\in\wt$ can be written in a unique way as
$$u=v+y\ \mbox{with}\ v\in V,\, y\in Y.$$
We consider the deformation $h:[0,1]\times (\wt\setminus D)\ri \wt\setminus D$ defined by
$$h(t,u)=(1-t)u+t\rho\,\frac{v}{\|v\|}\ \mbox{for all $t\in [0,1]$ , $u\in\wt\setminus D$}.$$

We have
$$h(0,u)=u\ \mbox{and}\ h(1,u)=\rho\,\frac{v}{\|v\|}\in V\cap\partial B_\rho=E_0.$$

It follows that $E_0$ is a deformation retract of $\wt\setminus D$ (see Papageorgiou, R\u adulescu and Repov\v s \cite[Definition 5.3.10, p. 385]{17}). Hence
$$i_*:H_0(E_0)\ri H_0(\wt\setminus\{0\})$$
is an isomorphism (see Eilenberg and Steenrod \cite[Theorem 11.5, p.30]{8} and Papageorgiou, R\u adulescu and Repov\v s \cite[Remark 6.1.6, p. 460]{18}).

The set $E=V\cap B_\rho$ is contractible (it is an interval). Hence $H_0(E,E_0)=0$ (see Eilenberg and Steenrod \cite[Theorem 11.5, p. 30]{8}). Therefore, if $j_*:H_0(E_0)\ri H_0(E)$, then ${\rm dim\,im}\,j_*=1$ (see Papageorgiou, R\u adulescu and Repov\v s \cite[Remark 6.1.26, p. 468]{8}). So, finally we have
$$\begin{array}{ll}
&\di {\rm dim\,im}\, i_*-{\rm dim\,im}\, j_*=2-1=1,\\
\Rightarrow &\di \varphi (\cdot)\ \mbox{has a local $(1,1)$-linking at 0, see Definition \ref{defi3}.}
\end{array}$$
The proof is now complete.
\end{proof}

From Proposition \ref{prop5} and Theorem 6.6.17 of Papageorgiou, R\u adulescu and Repov\v s \cite[p. 538]{18}, we have
\bb\label{eq34} {\rm dim}\, C_1(\varphi, 0)\geq 1.\bbb
Moreover, Proposition 3.9 of Papageorgiou, R\u adulescu and Repov\v s \cite{17} leads to the following result.

\begin{prop}\label{prop6} If hypotheses $H_0,\,H_1$ hold, then $C_k(\varphi,\infty)=0$ for all $k\in\NN_0$.\end{prop}

We are now ready for the existence theorem concerning the superlinear case.

\begin{theorem}\label{th7} If hypotheses $H_0,\,H_1$ hold, then problem \eqref{eq1} has a nontrivial solution $u_0\in \wt\cap L^\infty(\Omega)$.\end{theorem}

\begin{proof} On account of \eqref{eq34} and Proposition \ref{prop6}, we can apply Proposition 6.2.42 of Papageorgiou, R\u adulescu and Repov\v s \cite[p. 499]{18}. So, we can find $u_0\in\wt$ such that
$$\begin{array}{ll}
&\di u_0\in K_\varphi \setminus\{0\},\\
\Rightarrow &\di u_0\in \wt\cap L^\infty(\Omega)\ \mbox{is a solution of problem \eqref{eq1}, see \cite[Section 3.2]{18}}.\end{array}$$
The proof is now complete.
\end{proof}

\section{Resonant case}
In this section we are concerned with the resonant case ($p$-linear case). Our hypotheses allow resonance at $\pm\infty$ with respect to the principal eigenvalue $\hat\lambda_1(p)>0$.

The new conditions on the reaction $f(z,x)$ are the following.

\smallskip
$H_2$: $f:\Omega\times\RR\ri\RR$ is a Carath\'eodory function such that $f(z,0)=0$ for a.a. $z\in\Omega$ and\\
(i) $|f(z,x)|\leq \hat a(z)(1+|x|^{r-1})$ for a.a. $z\in\Omega$, all $x\in\RR$, with $\hat a\in L^\infty(\Omega)$, $p<r<q^*$;\\
(ii) if $F(z,x)=\int_0^xf(z,s)ds$, then $\lim_{x\ri\pm\infty}pF(z,x)/|x|^p\leq\hat\lambda_1(p)$ uniformly for a.a. $z\in\Omega$;\\
(iii) we have
$$f(z,x)x-pF(z,x)\ri +\infty\ \mbox{uniformly for a.a. $z\in\Omega$, as $x\ri\pm\infty$};$$
(iv) there exist $\delta>0$, $\theta\in L^\infty(\Omega)$ and $\hat\lambda>0$ such that
$$0\leq\theta (z)\ \mbox{for a.a. $z\in\Omega$, $\theta\not\equiv 0$, $\hat\lambda\leq \hat\lambda_2(q)$},$$
$$\theta(z)|x|^q\leq q F(z,x)\leq \hat\lambda |x|^q\ \mbox{for a.a. $z\in\Omega$ and all $|x|\leq\delta$}.$$

\begin{remark}\label{rem4}
Hypothesis $H_2(ii)$ implies that at $\pm\infty$, we can have resonance with respect to the principal eigenvalue of the operator $u\mapsto -{\rm div}\, (a_0(z)|Du|^{p-2}Du)-\Delta_qu$ with Robin boundary condition.
\end{remark}

\begin{prop}\label{prop8}
If hypotheses $H_0,\, H_2$ hold, then the energy functional $\varphi (\cdot)$ is coercive.
\end{prop}

\begin{proof}
We have
$$\begin{array}{ll}
\di \frac{d}{dx}\left(\frac{F(z,x)}{|x|^p}\right)&\di=\frac{f(z,x)|x|^p-p|x|^{p-2}xF(z,x)}{|x|^{2p}}\\
&\di =\frac{|x|^{p-2}x[f(z,x)x-pF(z,x)]}{|x|^{2p}}\,.\end{array}$$
On account of hypothesis $H_2(iii)$, given any $\gamma>0$, we can find $M_1=M_1(\gamma)>0$ such that
$$f(z,x)x-pF(z,x)\geq\gamma\ \mbox{for a.a. $z\in\Omega$ and all $|x|\geq M_1$}.$$
Hence we obtain
$$\frac{d}{dx}\left(\frac{F(z,x)}{|x|^p}\right)\left\{\begin{array}{lll}
&\di\geq\frac{\gamma}{x^{p+1}}&\quad \mbox{if}\ x\geq M_1\\
&\di\leq -\frac{\gamma}{|x|^{p+1}}&\quad\mbox{if}\ x\leq-M_1.\end{array}\right.
$$
Integrating, we obtain
\bb\label{eq35}\frac{F(z,x)}{|x|^p}-\frac{F(z,x)}{|u|^p}\geq-\frac{\gamma}{p}\left(\frac{1}{|x|^p}-\frac{1}{|u|^p}\right)\ \mbox{ for a.a. $z\in\Omega$ and all $|x|\geq |u|\geq M_1$.}\bbb

On account of hypothesis $H_2(ii)$, given $\ep>0$, we can find $M_2=M_2(\ep)>0$ such that
$$F(z,x)\leq\frac 1p\,(\hat\lambda _1(p)+\ep)|x|^p\ \mbox{for a.a. $z\in\Omega$ and all $|x|\geq M_2$.}$$
Using this inequality in \eqref{eq35} and letting $|x|\ri\infty$ we obtain
\bb\label{eq36}
\begin{array}{ll}
&\di \frac 1p\,(\hat\lambda_1(p)+\ep)-\frac{F(z,u)}{|u|^p}\geq\frac\gamma p\,\frac{1}{|u|^p}\ \mbox{for a.a. $z\in\Omega$ and all $|u|\geq M=\max\{M_1,M_2\}$},\\
\Rightarrow &\di (\hat\lambda_1(p)+\ep)|u|^p-pF(z,u)\geq\gamma\ \mbox{for a.a. $z\in\Omega$ and all $|u|\geq M$.}\end{array}\bbb

Arguing by contradiction, suppose that $\varphi(\cdot)$ is not coercive. Then we can find $\{u_n\}_{n\geq 1}\subseteq\wt$ such that
\bb\label{eq37} \|u_n\|\ri\infty\ \mbox{and}\ \varphi (u_n)\leq M_0\ \mbox{for some $M_0>0$ and all $n\in\NN$.}\bbb

Let $y_n=u_n/\|u_n\|$ for all $n\in\NN$. Then $\|y_n\|=1$, hence we may assume that
\bb\label{eq38} y_n\xrightarrow{w} y\ \mbox{in}\ \wt\ \mbox{and} \ y_n\ri y\ \mbox{in $L^p(\Omega)$ and in $L^p(\partial\Omega)$.}\bbb
From \eqref{eq37} we have
$$\begin{array}{ll}
&\di\frac 1p\,\gamma_p(y_n)+\frac 1q\, \frac{1}{\|u_n\|^{p-q}}\intom |Dy_n|^qdz-\intom \frac{F(z,u_n)}{\|u_n\|^p}dz\leq\frac{M_0}{\|u_n\|^p},\\
\Rightarrow &\di \gamma_p(y_n)+\frac pq\, \frac{1}{\|u_n\|^{p-q}}\intom |Dy_n|^qdz\leq \tau_n+(\hat\lambda_1(p)+\ep)\,\|y_n\|^p_p\ \mbox{with $\tau_n\ri 0$, see \eqref{eq36}},\\
\Rightarrow &\di \gamma_p(y)\leq(\hat\lambda_1(p)+\ep)\, \|y\|^p\ \mbox{(see \eqref{eq38})},\\
\Rightarrow &\di \gamma_p(y)\leq\hat\lambda_1(p)\|y\|^p_p\ \mbox{(since $\ep>0$ is arbitrary)},\\
\Rightarrow &\di y=\mu\hat u_1(p)\ \mbox{for some $\mu\in\RR$ (see \eqref{eq4})}.\end{array}$$

If $\mu=0$, then $y=0$ and so $\gamma_p(y_n)\ri0$. Hence, as in the proof of Proposition \ref{prop4}, we have $y_n\ri 0$ in $\wt$, contradicting the fact that $\|y_n\|=1$ for all $n\in\NN$.

So, $\mu\not=0$ and since $\hat u_1(p)(z)>0$ for a.a. $z\in\Omega$, we have $|u_n(z)|\ri+\infty$ for a.a. $z\in\Omega$. By \eqref{eq37} and \eqref{eq4} we have
\bb\label{eq39} \intom \left[ \frac 1p\,\hat\lambda_1(p)|u_n|^p-F(z,u_n)\right]dz\leq M_0\ \mbox{for all}\ n\in\NN.\bbb
However, from \eqref{eq36} and since $\gamma>0$ is arbitrary, we can infer that
\bb\label{eq40} \begin{array}{ll}
&\di\frac 1p\,\hat\lambda_1(p)|u_n|^p-F(z,u_n)\ri+\infty\ \mbox{for a.a. $z\in\Omega$, as $n\ri\infty$},\\
\Rightarrow &\di \intom\left[\frac 1p\, \hat\lambda_1(p)|u_n|^p-F(z,u_n)\right]dz\ri+\infty\ \mbox{by Fatou's lemma}.\end{array}\bbb

Comparing \eqref{eq39} and \eqref{eq40} we arrive at a contradiction. Therefore we can conclude that $\varphi(\cdot)$ is coercive.
\end{proof}

Using Proposition \ref{prop8} and Proposition 5.1.15 of Papageorgiou, R\u adulescu and Repov\v s \cite[p. 369]{18}, we obtain the following result.

\begin{corollary}\label{cor9} If hypotheses $H_0,\,H_2$ hold, then the energy functional $\varphi(\cdot)$ is bounded below and satisfies the C-condition.
\end{corollary}

Now we are ready for the multiplicity theorem in the resonant case.

\begin{theorem}\label{th10}
If hypotheses $H_0,\,H_2$ hold, then problem \eqref{eq1} has at least two nontrivial solutions $u_0,\,\hat u\in\wt\cap L^\infty(\Omega)$.
\end{theorem}

\begin{proof}
By Proposition \ref{prop5} we know that $\varphi(\cdot)$ has a local $(1,1)$-linking at the origin. Note that for that result mattered only the behavior of $f(z,\cdot)$ near zero and this is common in hypotheses $H_1$ and $H_2$. Also, we know that $\varphi(\cdot)$ is sequentially weakly lower semicontinuous. This fact in conjunction with Proposition \ref{prop8}, permit the use of the Weierstrass-Tonelli theorem. So, we can find $u_0\in\wt$ such that
\bb\label{eq41}\varphi(u_0)=\min\{\varphi(u):\ u\in\wt\}.\bbb
On account of hypothesis $H_2(iv)$ and since $q<p$,  we have
$$\begin{array}{ll}
&\di \varphi (u_0)<0=\varphi(0),\\
\Rightarrow &\di u_0\not=0\ \mbox{and}\ u_0\in K_\varphi,\\
\Rightarrow &\di u_0\in K_\varphi\cap L^\infty (\Omega)\ \mbox{is a nontrivial solution of \eqref{eq1}}.
\end{array}$$
Moreover, by Corollary 6.7.10 of Papageorgiou, R\u adulescu and Repov\v s \cite[p. 552]{18}, we can find $\hat u\in K_\varphi$, $\hat u\not\in\{0,u_0\}$. Then $\hat u\in\wt\cap L^\infty(\Omega)$ is the second nontrivial solution of problem \eqref{eq1}.
\end{proof}

\medskip
{\bf Acknowledgments.} The authors wish to thank a knowledgeable referee for  corrections and remarks. This research was supported by the Slovenian Research Agency grants
P1-0292, J1-8131, N1-0114, N1-0064, and N1-0083.

\end{document}